\documentclass[12pt]{amsart}

\usepackage{color}
\usepackage{graphicx}
\usepackage{amsmath}
\usepackage{amsfonts}
\usepackage{amssymb}
\usepackage{xypic}
\usepackage[all]{xy}
\usepackage{array}
\usepackage{enumerate}
\usepackage{url}
\usepackage{tikz}
\usetikzlibrary{matrix,arrows}

\newtheorem{theorem}{Theorem}[section]

\newtheorem{corollary}[theorem]{Corollary}

\newtheorem{lemma}[theorem]{Lemma}

\usepackage{xcolor}
\usepackage{graphicx}
\usepackage{amsmath}
\usepackage{amsfonts}
\usepackage{amssymb}
\usepackage{xypic}
\usepackage[all]{xy}
\usepackage{array}
\usepackage{enumerate}
\usepackage{url}

\newcommand{\DD}{\mathcal{D}}

\newcommand{\LL}{\mathcal{L}}
\newcommand{\RR}{\mathcal{R}}

\title{Regular Antilattices}
\author{Karin Cvetko-Vah \and 
Michael Kinyon \and  
Jonathan Leech \and  
Toma\v z Pisanski
}
\address[Cvetko-Vah]{Department of Mathematics \\
Faculty of Mathematics and Physics \\
University of Ljubljana \\ Jadranska 21, SI-1000 Ljubljana, Slovenia}

\address[Kinyon]{Department of Mathematics \\
University of Denver \\ Denver, CO 80208, USA}

\address[Leech]{Department of Mathematics \\
Westmont College \\ Santa Barbara, CA 93108, USA}

\address[Pisanski]{Faculty of Mathematics, Natural Sciences
and Information Technologies \\  
University of Primorska \\
Glagolja\v{s}ka 8 \\ SI-6000 Koper, Slovenia \\
and \\
Andrej Maru\v{s}i\v{c} Institute \\
University of Primorska \\
Muzejsi trg 2 \\ 6000 Koper, Slovenia} 
\date{\today}
\begin{document}
\maketitle
\begin{abstract} 
Antilattices $(S;\lor, \land)$ for which the Green's equivalences $\LL_{(\lor)}$, $\RR_{(\lor)}$, $\LL_{(\land)}$ and $\RR_{(\land)}$ are all congruences of the entire antilattice are studied and enumerated.
\end{abstract}



\bigskip

\section{Introduction}
In the study of noncommutative  lattices, lattices still play an
important role. They are the commutative cases of the algebras being considered and
indeed play an important role in the general theory of that larger class of algebras. (As with rings, ``noncommutative'' is understood inclusively to mean \emph{not necessarily commutative}). But
also, typically, a second subclass of algebras exists that plays counterpoint to the subclass
of lattices. It has become common to refer to  their members  as
``antilattices.'' Typically they resist any kind of nontrivial commutative behavior. That
is, an instance of $xy = yx$ for a relevant binary operation can occur only when $x = y$. 
Antilattices, however, are not without their special charm. Indeed, they have been studied
in connection with magic squares and finite planes. (See \cite{Le05}.) In this paper we study the
class of regular antilattices (defined below). Basic  concepts, e.g. bands, quasilattices and preliminary  results of relevance are assembled in
the first section. Regular antilattices themselves are the focus of the next section. A closer look at their lattice of subvarieties occurs in the final section. The reader
seeking further  information on bands is referred to the presentations given in
Clifford and Preston \cite{Cl61} and in Howie \cite{Ho77}. For further background on skew lattices and
quasilattices,  see \cite{La02} and \cite{Le19}. The basic facts of universal algebras, and in particular varieties, may be found in the second chapter of \cite{Bu81}.

\section{Preliminary Concepts and Results}

A \emph{band} is a semigroup $(S; \cdot)$ for which all elements are idempotent, that is, $xx = x$
holds. A band is \emph{rectangular} if it satisfies the identity $xyz = xz$, or equivalently, $xyx = x$.
(As often occurs, if just a single binary operation is involved, its appearance is suppressed
in equations.) A \emph{semilattice} is a commutative band ($xy = yx$). Clearly, rectangular
semilattices form the class of trivial $1$-point bands. Indeed both classes are structural
opposites that play important roles in the general structure of bands. To see how and to
set the stage for further preliminaries requires the use of Green's relations, defined first
for bands.
\[
\begin{array}{llll}
\DD: & x\,\DD\,y & \text{ iff } &  \text{both } xyx=x  \text{ and } yxy = y; \\
\LL: & x\,\LL\,y & \text{ iff } &  \text{both } xy=x  \text{ and } yx = y; \\
\RR: & x\,\RR\,y & \text{ iff } &  \text{both } xy=y  \text{ and } yx = x. \\
\end{array}
\]
For bands, $\LL$ and $\RR$ commute under the usual composition of relations, with the common
outcome being $\DD$, i.e., $\LL\circ \RR = \RR\circ \LL = \RR\lor \LL = \DD$. Here $\RR \lor \LL$ denotes the join of the two
relations. Moreover, we have the following fundamental result of Clifford and McLean:

\begin{theorem}  Given a band $(S; \cdot)$, the relation $\DD$ is a congruence for which $S/\DD$ is the maximal semilattice image and each $\DD$-class of $S$ is a maximal rectangular subalgebra of $S$.  In brief, every band is a semilattice of rectangular bands.
\end{theorem}

	So what do rectangular bands look like?  First there are two basic cases.  A \emph{left-zero band} is a band $(L; \cdot)$ with the trivial composition: $xy = x$.  A \emph{right-zero band} is a band $(R; \cdot)$ with the trivial composition: $xy = y$.  In other words, we either have just a single $\LL$-class or just a single $\RR$-class.  Finally, there are products of both types, $L\times R$,  and to within isomorphism, that is it.  A finite rectangular band does form a rectangular grid consisting of rows that are $\RR$-classes and columns that are $\LL$-classes.

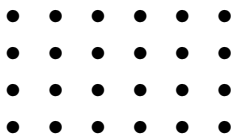
\[
\begin{array}{cccccc}
\bullet & \bullet & \bullet & \bullet & \bullet & \bullet \\
\bullet & \bullet & \bullet & \bullet & \bullet & \bullet \\
\bullet & \bullet & \bullet & \bullet & \bullet & \bullet\\
\bullet & \bullet & \bullet & \bullet & \bullet & \bullet
\end{array}
\]

The product $xy$ of elements $x$ and $y$ is the unique element in the row of $x$ and the column of $y$. Given a rectangular band $(S;\cdot)$ and $x$ in $S$, if $L$ denotes the $\LL$-class of $x$ and $R$ denotes the $\RR$-class of $S$, and $\varphi:L\times R\to S$ is defined by $\varphi(u,v)=uv\in S$, then $\varphi$ is an isomorphism of rectangular bands. Rectangular bands are precisely the bands that are \emph{anti-commutative} in that $xy=yx$ iff $x=y$.

While bands have a very simple local structure -- their rectangular $\DD$-classes --
it is not immediately clear how elements from different $\DD$-classes combine under
the binary operation.

A band is \emph{regular} if the relations $\LL$ and $\RR$ are both congruences.  Semilattices and rectangular bands are both regular.  In the semilattice case $\LL$ and $\RR$ reduce to the identity relation, so that regularity is trivial.  One might expect all bands to be regular, but that is not so.  In the rectangular case there is more: $\LL$ and $\RR$ commute under composition, not only with each other, but with every congruence $\theta$:  
\[
\LL\circ\theta = \theta\circ\LL = \theta\lor\LL \text{ and } \RR\circ\theta = \theta\circ\RR = \theta\lor\RR.
\]
 
	A \emph{double band} is an algebra $(S; \lor, \land)$ for which both reducts $(S; \lor)$ and $(S; \land)$ are bands.  A lattice is thus a double band where both $(S; \lor)$ and $(S; \land)$ are semilattices that jointly satisfy the standard \emph{absorption} identities for a lattice: $x\land(x\lor y) = x = x\lor(x\land y)$.  A very general class of noncommutative lattices is as follows.  A \emph{quasilattice} is a double band that satisfies the following (modified) absorption identities: 
\[
x\land(y\lor x\lor y)\land  x=  x  =  x\lor(y\land x\land y)\lor x.
\]

Note that if commutativity is assumed, both identities reduce to the absorption identities for a lattice.  

A \emph{skew lattice} is a noncommutative lattice that satisfies the dual absorption identities: 

\[
x\land(x\lor y)=  x  =  (x\lor y)\land x,
\]
\[
x\lor(x\land y)=  x  =  (x\land y)\lor x.
\]

A {skew lattice} is a quasilattice, but not conversely.  In a quasilattice, both operations share common $\DD$-classes that also form subalgebras, although on these classes both operations need not agree!  Clearly, for a quasilattice $(S; \lor, \land)$, $\DD$ is a congruence. Indeed, $S/\DD$ is the maximal lattice image of $S$.   This leads us to:
 
	An \emph{antilattice} is a double band $(S; \lor, \land)$ for which both reducts, $(S; \lor)$ and $(S; \land)$, are rectangular bands, i.e., satisfy the identity $xyz = xz$ or equivalently $xyx = x$.  An antilattice is trivially a quasilattice.  Conversely, each $\DD$-class of a quasilattice is a subalgebra that is an antilattice. If the antilattice is a skew lattice, it is also called a \emph{rectangular skew lattice}.  As an antilattice, it is characterized by $x\land y = y\lor x$.
 $\DD$-classes of skew lattices are always {rectangular skew lattices}. Similar to bands, a version of the Clifford-McLean Theorem holds: 

\begin{theorem}  Given a quasilattice $(S; \lor, \land)$, the relation $\DD$ is the same for both, $(S; \lor)$ and $(S; \land)$.  For this common congruence, the quotient algebra $S/\DD$ is the maximal lattice image and each $\DD$-class of S is a maximal sub-antilattice of $S$.  In brief, every quasilattice is a lattice of antilattices. (Compare Corollary 3 of \cite{La02}; see also \cite{Le19}.)
\end{theorem}

	Antilattices have been studied, not only due to their connection to quasilattices, but also in connection with magic squares and finite plains.  (See \cite{Le05}.)  
	
	Like quasilattices and semigroups, by definition antilattices do not have prescribed constants, thus making the empty set a viable subalgebra, but in so doing allow for the existence of a complete lattice of subalgebras for any given antilattice.
	
\section{Regular Antilattices}

	Given an antilattice $(S; \lor, \land)$, both reducts $(S; \lor)$ and $(S; \land)$ are {regular} in that $\LL_{(\lor)}$ and $\RR_{(\lor)}$ are congruences on $(S; \lor)$, and likewise $\LL_{(\land)}$ and $\RR_{(\land)}$ are congruences on $(S; \land)$.  The antilattice is \emph{regular} if all four equivalences are congruences for the whole algebra.  In general, a quasilattice $(S; \lor, \land)$ is \emph{regular} if $\LL_{(\lor)}$, $\RR_{(\lor)}$, $\LL_{(\land)}$ and $\RR_{(\land)}$ are congruences of $(S; \lor, \land)$.  Skew lattices are regular, but in general, quasilattices need not be regular.

\begin{theorem} Regular antilattices form a subvariety of the variety of antilattices.
\end{theorem}

\begin{proof} We show that antilattices for which $\LL_{(\lor)}$ is a congruence form a subvariety.  To begin, in an antilattice $S$, $x \,\LL_{(\lor)}\, u\lor x$ holds for all $u, x\in S$, and conversely, if $x\,\LL_{(\lor)} \, x'$, then trivially, $x'= x'\lor x$.  Since $\LL_{(\lor)}$ is already a congruence on the reduct $(S; \lor)$, for $\LL_{(\lor)}$ to be a congruence on $(S; \land)$, precisely the following identities need to hold:
\[
	(y\land x) \lor [y\land(u\lor x)] = y\land x \, \, \& \,\,   [y\land(u\lor x)] \lor (y\land x) = y\land(u\lor x)
\]
	and
\[
		(x\land y) \lor [(u\lor x)\land y] = x\land y \, \, \& \,\,   [(u\lor x)\land y] \lor (y\land x) = (u\lor x)\land y.
\]
Thus this class of antilattices indeed forms a subvariety.  Similar remarks verify the same claim for $\RR_{(\lor)}$, $\LL_{(\land)}$ and $\RR_{(\land)}$.  The theorem now follows.
\end{proof}

Since $\LL$ and $\RR$ commute under composition with all congruences on a rectangular band, $\LL_{(\lor)}\circ\LL_{(\land)} = \LL_{(\land)}\circ \LL_{(\lor)}$; $\LL_{(\lor)}\circ \RR_{(\land)} = \RR_{(\land)}\circ \LL_{(\lor)}$; $\RR_{(\lor)}\circ \LL_{(\land)} = \LL_{(\land)}\circ \RR_{(\lor)}$; and $\RR_{(\lor)}\circ \RR_{(\land)} = \RR_{(\land)}\circ\RR_{(\lor)}$ hold for regular antilattices.  All four outcomes are thus congruences on the antilattice, and indeed form the join congruences of the respective pairs of congruences. 

	An antilattice $(S; \lor, \land)$ is \emph{flat} if the reduct $(S; \lor)$ is either a left-$0$ semigroup ($x\lor y = x$) or a right-$0$ semigroup ($x\lor y = y$), and likewise the reduct $(S; \land)$ is either a left-$0$ semigroup ($x\land y = x$) or a right-$0$ semigroup ($x\land y = y$).  That is, for each operation, either $\DD = \LL$ or $\DD = \RR$.  Clearly there are $4$ distinct classes of flat antilattices: 
\[\begin{array}{l}
	\text{the class } \mathbf{A}_{\LL \LL} \text{ of all antilattices where } x\lor y = x = x\land y.  \\
	\text{the class } \mathbf{A}_{\LL \RR} \text{ of all antilattices where } x\lor y = x \text{ but } x\land y = y.\\
	\text{the class } \mathbf{A}_{\RR \LL} \text{ of all antilattices where } x\lor y = y \text{ but } x\land y = x. \\
	\text{the class } \mathbf{A}_{\RR \RR} \text{ of all antilattices where } x\lor y = y = x\land y.
\end{array}
\]
Clearly \emph{each class is a subvariety of regular antilattices}.  What is more: 

\begin{lemma} Flat antilattices $S$ and $T$ of the same class are isomorphic if and only if they have the same cardinality.  When the latter is the case, an isomorphism is given by any bijection between $S$ and $T$.
\end{lemma}

\begin{theorem}[Decomposition Theorem] Every nonempty regular antilattice $(S; \lor, \land)$ factors into the direct product $S_{\LL \LL}\times  S_{\LL \RR} \times  S_{\RR \LL} \times  S_{\RR \RR}$ of its four maximal flat images, one from each class above, with the respective factors being unique to within isomorphism.
\end{theorem}

\begin{proof} The factorization is obtained by first factoring with respect to say $\lor$:
\[
S \cong S/\RR_{(\lor)} \times S/\LL_{(\lor)}
\]
to get two factors for which the $\lor$-operation is flat.  Then similarly factor both factors further with respect to the relevant $\RR_{(\land)}$ and $\LL_{(\land)}$ congruences to get four flat factors:
\[
S\cong S_{\LL \LL}\times  S_{\LL \RR} \times  S_{\RR \LL} \times  S_{\RR \RR}
\]
where: $S_{\LL \LL}= S/(\RR_{(\lor)}\lor\RR_{(\land)})$; $S_{\LL \RR} = S/(\RR_{(\lor)}\lor\LL_{(\land)})$; $S_{\RR \LL} = S/(\LL_{(\lor)}\lor\RR_{(\land)})$ and $S_{\RR \RR} = S/(\LL_{(\lor)}\lor\LL_{(\land)})$.
\end{proof}


Further factorization can take place.  But first, given a positive integer $n$, let $\mathbf{n}_{\LL  \LL}$, $\mathbf{n}_{ \LL  \RR}$, $\mathbf{n}_{ \RR  \LL}$ and $\mathbf{n}_{ \RR  \RR}$ denote the relevant flat antilattices on the set $\{1, 2, 3,\dots, n\}$.  This leads us to the following finite version of the Decomposition Theorem:

\begin{theorem}\label{th:fact-n} Let $(S; \lor, \land)$ be a nonempty finite regular antilattice with the above factorization $S_{\LL \LL} \times S_{\LL  \RR} \times  S_{\RR \LL} \times S_{\RR \RR}$.  If ${n}_{ \LL  \LL} = |S_{\LL \LL }|$, ${n}_{ \LL  \RR} = |S_{\LL  \RR} |$, etc., then 
\[
S  \cong  \mathbf{n}_{ \LL  \LL} \times \mathbf{n}_{ \LL  \RR} \times \mathbf{n}_{ \RR  \LL} \times \mathbf{n}_{ \RR  \RR}.
\]
\end{theorem}

Clearly these four parameters characterize $(S; \lor, \land)$.  It is also clear that factorization can continue on each of the four factors.  For instance say $|S_{\LL \LL} | = 180 = 4\times 5\times 9$.  Then we have $S_{\LL \LL} \cong (\mathbf 2_{\LL \LL})^2 \times (\mathbf 3_{\LL \LL})^2\times  \mathbf 5_{\LL \LL}$.  For the trivial $1$-point algebra we just have $1$. 

\begin{corollary}\label{cor:dir-irr} A regular antilattice $(S; \lor, \land)$ is directly irreducible iff $|S|$ is either $1$ or a prime.  Every finite regular antilattice of order $>1$ thus factors into a direct product  of finitely many flat antilattices of prime order.
\end{corollary}

\begin{corollary}  A regular antilattice $(S; \lor, \land)$ is subdirectly irreducible iff either $|S| = 1$ or $|S| = 2$.  Every  (finite) regular antilattice of order $>1$ is thus isomorphic to a subdirect product of (finitely) many flat antilattices of order $2$.
\end{corollary}

A sub-(pseudo)variety of regular antilattices is \emph{positive}, if it is not the sub-variety $\{\emptyset\}$.

\begin{corollary}\label{cor:lattice-var} The lattice of all positive subvarieties of regular antilattices is a Boolean algebra with $16$ elements and $4$ atoms: $\mathbf A_{\LL \LL}$, $\mathbf A_{\LL  \RR}$, $\mathbf A_{\RR \LL}$ and $\mathbf A_{\RR \RR}$. 
\end{corollary}

\begin{corollary} The lattice of all positive sub-pseudovarieties of finite regular antilattices is a Boolean algebra with $16$ elements and the four atoms as above, but with their respective classes now restricted to finite algebras: ${_f}\mathbf A_{\LL \LL}$, ${_f}\mathbf A_{\LL  \RR}$, ${_f}\mathbf A_{\RR \LL}$ and ${_f}\mathbf A_{\RR \RR}$.
\end{corollary}

We will take a closer look at the positive subvarieties involved in the third section.

\emph{What can one say about the congruence lattice of a regular antilattice?} To begin
observe that the four classes of flat antilattces are mutually term equivalent with each
other and with the class of all left-zero semigroups and also the class of all right-zero
semigroups. In all these special cases the congruence lattice is precisely the full lattice
$\Pi(S)$ of all equivalences of the underlying set $S$. Following the situation for rectangular
bands in general, we have:

\begin{theorem} Let a nonempty regular antilattice $(S; \lor, \land)$ be factored into the direct product of its
four maximal flat images: $S_{\LL \LL}\times  S_{\LL \RR} \times  S_{\RR \LL} \times  S_{\RR \RR}$. Then the congruence lattice of $S$ is
given by $\Pi(S_{\LL \LL})\times  \Pi(S_{\LL \RR}) \times  \Pi(S_{\RR \LL}) \times  \Pi(S_{\RR \RR})$. That is, if the elements if $S$ are expressed
as $4$-tuples $(x, y, z, w)$ given by the factorization, then each congruence $\theta$ on $S$ can be
represented as a $4$-tuple $(\theta_{\LL \LL}, \theta_{\LL \RR}, \theta_{\RR \LL}, \theta_{\RR \RR})$ of congruences on each factor in that:
\[
(x, y, z, w) \,\theta\, (x', y', z', w')  \text{ iff } x\,\theta_{\LL \LL}\,x', y\,\theta_{\LL \RR}\,y', z\,\theta_{\RR \LL}\,z' \,\&\, w\,\theta_{\RR \RR}\,w'.
\]
Conversely, in this manner every such $4$-tuple of congruences defines a congruence on
the full antilattice $S$.
\end{theorem}

In similar fashion:

\begin{theorem} Given a nonempty regular antilattice $S$ with factorization $S_{\LL \LL}\times  S_{\LL \RR} \times  S_{\RR \LL} \times  S_{\RR \RR}$, if
$a = |S_{\LL \LL}|$, $b = |S_{\LL \RR}|$, $c = |S_{\RR \LL}|$ and $d = |S_{\RR \RR}|$, then the number of subalgebras of $S$ is:
\[
1 + (2^a-1)(2^b-1)(2^c-1)(2^d-1).
\]
\end{theorem} 
One can ask: \emph{given a positive integer $n\geq 1$, up to isomorphism, how many nonisomorphic
regular antilattices are there of size $n$?} By Theorem \ref{th:fact-n} it is the number $\rho(n)$
of $4$-fold positive factorizations $abcd$ of $n$, where the order of the factors $a$, $b$, $c$, $d$ is
important. Here $a$ is the size of the $\LL \LL$-factor, $b$ is the size of the $\LL \RR$-factor, etc.

%

To begin, thanks to Corollary \ref{cor:dir-irr}, given the prime
power factorization $n=2^{e_2} 3^{e_3} 5^{e_5}\dots p_k^{e_{p_k}}$:
\[
\rho (n)=\rho(2^{e_2}) \rho(3^{e_3})\rho(5^{e_5})\dots \rho(p_k^{e_{p_k}}).
\]
Thus things can be reduced to calculating $\rho(p^e)$ for any prime power $p^e$ .

From a combinatorial perspective, this is equivalent to asking in how many distinct ways can $e$
identical balls  be distributed into $4$ labeled boxes. This question has a simple answer:

$$\rho(p^e) = {{e+3}\choose{3}}.$$

By putting all these together
we obtain the following closed formula for  $\rho (n)$:

\begin{theorem}
Let $n$ have the following  prime
power factorization $n=2^{e_2} 3^{e_3} 5^{e_5}\dots p_k^{e_{p_k}}$:
Then 
\[
\rho (n)= {{e_2+3}\choose{3}} {{e_3+3}\choose{3}} {{e_5+3}\choose{3}} \dots  {{e_{p_k}+3}\choose{3}}.
\]

\end{theorem}

One can ask a more general question:
In how many distinct ways can $e$ identical balls  be distributed into $k$ labeled boxes? 
This question has analogous answer, namely  ${{e+k-1}\choose{k-1}}$; see, for instance \cite{Wikipedia}.
We will use some of these in the next section.
Note that such ordered partition of integer $n$ into $k$ possibly empty parts is sometimes called \emph{composition}; see \cite{Wikipedia}.
%

\section{Semi-flat antilattices and other subvarieties}

An antilattice $(S; \lor, \land)$ is \emph{semi-flat} if either $(S; \lor)$ or $(S; \land)$ is flat. Flat antilattices
are trivially semi-flat. The class of all semi-flat antilattices consists of four distinct
subclasses that are not necessarily disjoint:
\[\begin{array}{l}
\mathbf A_{\LL\#}, \text{ the class of all antilattices } (S; \lor, \land) \text{ s.t. } (S; \lor) \text{ is a left-}0 \text{ band} \\
\mathbf A_{\RR\#}, \text{ the class of all antilattices } (S; \lor, \land) \text{ s.t. } (S; \lor) \text{ is a right-}0 \text{ band} \\
\mathbf A_{\#\LL}, \text{ the class of all antilattices } (S; \lor, \land) \text{ s.t. } (S; \land) \text{ is a left-}0 \text{ band} \\
\mathbf A_{\#\RR}, \text{ the class of all antilattices } (S; \lor, \land) \text{ s.t. } (S; \land) \text{ is a right-}0 \text{ band} 
\end{array}
\]

\begin{theorem} These four classes are subvarieties of the variety of regular antialttices.
\end{theorem}

\begin{proof} First observe that each is at least a subvariety in the variety of all antilattices. We
show this for $\mathbf A_{\LL\#}$, the other cases being similar. The identity characterizing $\mathbf A_{\LL\#}$ in the
variety of antilattices is clearly $x\lor y = x$. Thus $\mathbf A_{\LL\#}$ is indeed a subvariety of antilattices.
To see that all semi-flat antilattices are regular, again we need only consider, say, $\mathbf A_{\LL\#}$.
So let $(S; \lor, \land)$ be an antilattice for which $(S; \lor)$ is a left zero-band. Thus $\LL_{(\lor)}$ is the
universal equivalence $\nabla$ on $S$, and thus trivially a congruence on $(S; \land)$ while $\RR_{(\lor)}$ is the
identity equivalence and thus again trivially a congruence on $(S; \land)$. Next consider $\LL_{(\land)}$
and $\RR_{(\land)}$. Being congruences on $(S; \land)$, they are at least equivalences on $S$. But all
equivalences on $S$ are congruences on the left zero-band $(S; \lor)$, and thus $\LL_{(\land)}$ and $\RR_{(\land)}$ are
congruences on $(S; \lor, \land)$.
\end{proof}

Consider next the following diagram.

\begin{center}
\begin{tikzpicture}
\matrix (m) [matrix of math nodes, row sep=3.5em, column sep=2em, text height=1.5ex, text depth=0.25ex] 
{ \mathbf A_{\LL \LL} & & \mathbf A_{\LL \RR} \\
\mathbf A_{\RR \LL} & & \mathbf A_{\RR \RR} \\};
\path[.] (m-1-1) edge node[above] {$\mathbf A_{\LL\#}$} (m-1-3);
\path[.] (m-2-1) edge node[below] {$\mathbf A_{\RR\#}$} (m-2-3);

\path[.] (m-1-1) edge node[left] {$\mathbf A_{\# \LL}$} (m-2-1);
\path[.] (m-1-3) edge node[right] {$\mathbf A_{\# \RR}$} (m-2-3);
\end{tikzpicture}
\end{center}

The four flat varieties occupy the middle rectangle. If two distinct flat varieties
are adjacent on this rectangle, their join variety is the semi-flat variety labeling the 
line between them. But if they are diagonal opposites, we have the following:
\[\begin{array}{ll}
\mathbf A_{\LL \LL} \lor \mathbf A_{\RR \RR} = & \text{the subvariety of antilattices for which } x\lor y = x\land y. \\
\mathbf A_{\LL \RR} \lor \mathbf A_{\RR \LL} = & \text{the subvariety of antilattices for which } x\lor y = y\land x.
\end{array}
\]
These are the antilattice subvarieties that are, respectively, skew* lattices or skew lattices.

Next are the four double joins. Consider $\mathbf A_{\LL \LL} \lor \mathbf A_{\LL \RR}  \lor \mathbf A_{\RR \LL} $. It consists of regular
antilattices for which the $\mathbf A_{\RR \RR}$-factor is trivial. Since $\lor$ and $\land$ are idempotent, this reduces
to no nontrivial $\mathbf A_{\RR \RR}$-subalgebra occurring in the given antilattice. More briefly, no copy
of $\mathbf 2_{\RR \RR}$ occurs as a subalgebra. This is guaranteed by the identity $x \lor (x\land y) = x$ (that is
equivalent to the implication: $u \,\RR_{(\land)} \,v\Rightarrow u \,\LL_{(\lor)} \,v$) along with its $\lor-\land$ dual. This subvariety is, of course, the
Boolean complement $\mathbf A_{\RR \RR}^C$ of $\mathbf A_{\RR \RR}$.  The three other double joins are treated similarly to obtain:
$$  \mathbf A_{\LL \LL} \lor \mathbf A_{\LL \RR} \lor \mathbf A_{\RR \RR} = \mathbf A_{\RR \LL}^C$$
$$  \mathbf A_{\LL \LL} \lor \mathbf A_{\RR \LL} \lor \mathbf A_{\RR \RR} = \mathbf A_{\LL \RR}^C$$
$$  \mathbf A_{\LL \RR} \lor \mathbf A_{\RR \LL} \lor \mathbf A_{\RR \RR} = \mathbf A_{\LL \LL}^C$$

Finally, above these four lies the full variety of all regular antilattices
and just below the four flat cases lies the variety of trivial $1$-point
algebras.    The resulting lattice of all subvarieties of regular antilattices is, of course, isomorphic to the lattice of all subsets of any $4$-element set, which brings us back to Corollary \ref{cor:lattice-var}.

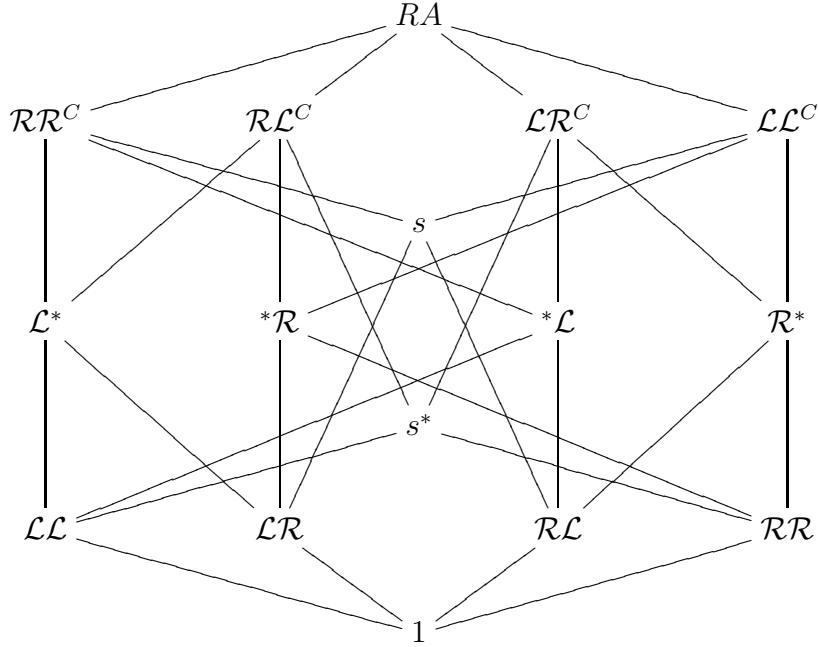
\begin{center}
\[
\xymatrix{
&&&& RA \ar@{-}[dlll]\ar@{-}[dl]\ar@{-}[dr]\ar@{-}[drrr]&&\\
& \RR\RR^C \ar@{-}[dd]\ar@{-}[ddrrrr]\ar@{-}[drrr] &&\RR\LL^C\ar@{-}[dddr]\ar@{-}[dd]\ar@{-}[ddll]  && \LL\RR^C\ar@{-}[dddl]\ar@{-}[ddrr]\ar@{-}[dd]  && \LL\LL^C\ar@{-}[dd]\ar@{-}[ddllll]\ar@{-}[dlll]\\
&&&& s &&\\
 &\LL^*  && ^*\RR && ^*\LL && \RR^* && \\
&&&& s^* &&\\
& \LL\LL \ar@{-}[uu]\ar@{-}[uurrrr]\ar@{-}[urrr] &&\LL\RR \ar@{-}[uuur]\ar@{-}[uull]\ar@{-}[uu] && \RR\LL \ar@{-}[uuul]\ar@{-}[uu]\ar@{-}[uurr]   && \RR\RR \ar@{-}[uu]\ar@{-}[uullll]\ar@{-}[ulll]\\
&&&& 1 \ar@{-}[ulll]\ar@{-}[ul]\ar@{-}[ur]\ar@{-}[urrr]&&\\
 }
\]
\end{center}

\bigskip

The Hasse diagram of the Boolean algebra is explained in the following table.

\bigskip

\bigskip
\begin{tabular}{l|l|l}
symbol & subvariety & number for $p^e$\\
\hline
\hline
$RA$ & regular antilattices& ${e+3}\choose{3}$\\
\hline
$\RR\RR^C$ & complement of $\RR\RR$& ${e+2}\choose{2}$\\ 
$\RR\LL^C$ & complement of $\RR\LL$& ${e+2}\choose{2}$\\ 
$\LL\RR^C$ & complement of $\LL\RR$& ${e+2}\choose{2}$\\ 
$\LL\LL^C$ & complement of $\LL\LL$& ${e+2}\choose{2}$\\
\hline 
$s$ & skew antilattices& ${{e+1}\choose{1}} = e+1$\\ 
$\LL^*$ & $\LL^*$ semi-flat& ${{e+1}\choose{1}} = e+1$\\ 
$^*\RR$ & $^*\RR$ semi-flat& ${{e+1}\choose{1}} = e+1$\\
$^*\LL$ & $^*\LL$ semi-flat& ${{e+1}\choose{1}} = e+1$\\
$\RR^*$ & $\RR^*$ semi-flat& ${{e+1}\choose{1}} = e+1$\\
$s^*$ & skew* antilattices& ${{e+1}\choose{1}} = e+1$\\
\hline
$\LL\LL$ & $\LL \LL$-flat& ${{e+0}\choose{0}} = 1$\\ 
$\LL\RR$ & $\LL \RR$-flat& ${{e+0}\choose{0}} = 1$\\ 
$\RR\LL$ & $\RR \LL$-flat& ${{e+0}\choose{0}} = 1$\\ 
$\RR\RR$ & $\RR \RR$-flat& ${{e+0}\choose{0}} = 1$\\ 
\hline
$1$ & trivial antilattice& $1$ if $e = 0$\\
\end{tabular}

\bigskip

Equipped with all necessary tools we may now perform enumeration of the pseudo-variety of finite regular antilattices 
and their sub-pseudo-varieties. These sequences can be found in OEIS \cite{OEIS}.

\bigskip

\begin{tabular}{c|c|c|c|c|c}
$n$ & RA & $\RR\RR^c, \RR\LL^c, \LL\RR^c,\LL\LL^c$ & s, $\LL^*$, $^*\RR$, $^*\LL$,$\RR^*$, s* & $\LL\LL, \LL\RR, \RR\LL, \RR\RR$ & 1 \\
\hline
OEIS &A007426&A007425&A000005&A000012&\\
\hline
1 &1&1&1&1&1\\
2 &4&3&2&1&0\\
3 &4&3&2&1&0\\
4 &10&6&3&1&0\\
5 &4&3&2&1&0\\
6 &16&9&4&1&0\\
7 &4&3&2&1&0\\
8 &20&10&4&1&0\\
9 &10&6&3&1&0\\
10 &16&9&4&1&0\\
11 &4&3&2&1&0\\
12 &40&18&6&1&0\\
13 &4&3&2&1&0\\
14 &16&9&4&1&0\\
15 &16&9&4&1&0\\
16 &35&15&5&1&0\\
\end{tabular}

\section*{Acknowledgements}
Michael Kinyon was supported by the Simons Foundation Collaboration Grant 359872.
Work of Toma\v z Pisanski is supported in part by the  ARRS grants P1-0294,J1-7051, N1-0032, and J1-9187.

\end{document}